\RequirePackage{etoolbox}
\csdef{input@path}{{style/}{graphics/}}
\makeatletter
\input{arxiv-vmsta.cfg}
\makeatother
\documentclass[numbers,compress]{vmsta}

\volume{1}
\pubyear{2014}
\firstpage{151}
\lastpage{165}
\doi{10.15559/15-VMSTA16}

\startlocaldefs

\allowdisplaybreaks

\newcommand{\rone}{\mathbb{R}}
\renewcommand{\Re}{\rone}

\DeclareMathOperator{\Ent}{\mathbf{Ent}}
\newcommand{\FF}{\mathbb{F}}
\newcommand{\Ff}{\mathcal{F}}
\newcommand{\ba}{\begin{aligned}}
\newcommand{\ea}{\end{aligned}}
\newcommand{\be}{\begin{equation}}
\newcommand{\ee}{\end{equation}}
\newcommand{\la}{\langle}
\newcommand{\ra}{\rangle}
\newcommand{\E}{\mathbf{E}}
\newcommand{\Nf}{\mathcal{N}}
\renewcommand{\P}{\mathbf{P}}

\newtheorem{thm}{Theorem}
\newtheorem{lem}{Lemma}
\newtheorem{prop}{Proposition}
\theoremstyle{definition}
\newtheorem{ex}{Example}

\endlocaldefs

\begin{document}

\begin{frontmatter}

\title{A martingale bound for the entropy associated with~a~trimmed
filtration on $\Re^d$}
\author[a]{\inits{A.}\fnm{Alexei}\snm{Kulik}\corref{cor1}}\email{kulik.alex.m@gmail.com}
\cortext[cor1]{Corresponding author.}

\author[b]{\inits{T.}\fnm{Taras}\snm{Tymoshkevych}}\email{tymoshkevych@gmail.com}

\address[a]{Institute of Mathematics, National Academy of Science,
Kiev, Ukraine}
\address[b]{Kyiv National Taras Shevchenko University, Ukraine}

\markboth{A. Kulik, T. Tymoshkevych}{A martingale bound for the entropy
associated with a trimmed filtration on $\Re^d$}

\begin{abstract}
Using martingale methods, we provide bounds for the entropy of a
probability measure on $\Re^d$ with the right-hand side given in a
certain integral form. As a~corollary, in the one-dimensional case, we
obtain a weighted log-Sobolev inequality.
\end{abstract}

\begin{keyword}
Martingale \sep entropy \sep log-Sobolev inequality
\sep trimmed regions \sep trimmed filtration
\MSC[2010]39B62 \sep47D07 \sep60E15 \sep60J60
\end{keyword}
\accepted{19 January 2015}
\revised{18 January 2015}
\received{22 December 2014}
\publishedonline{2 February 2015}
\end{frontmatter}

\section{Introduction}

A probability measure $\mu$ on $\Re^d$ is said to satisfy the \emph
{log-Sobolev inequality} if for every smooth compactly supported
function $f:\Re^d\to\Re$, the \emph{entropy} of~$f^2$, which by
definition equals
\[
\Ent_\mu f^2=\int_{\Re^d}f^2
\log f^2\, d\mu- \biggl(\int_{\Re^d}f^2\, d
\mu \biggr)\log \biggl(\int_{\Re^d}f^2\, d\mu
\biggr),
\]
possesses a bound
\be\label{LSI}
\Ent_\mu f^2 \leq2c \int_{\Re^d}\|\nabla f\|^2\, d\mu
\ee
with some constant $c$. The least possible constant $c$ such that (\ref
{LSI}) holds for every compactly supported smooth $f$ is called the
log-Sobolev constant for the measure $\mu$; the multiplier 2 in (\ref
{LSI}) is chosen in such a way that for the standard Gaussian measure
on $\Re^d$, its log-Sobolev constant equals 1.

The \emph{weighted log-Sobolev inequality} has the form
\be\label{WLSI}
\Ent_\mu f^2 \leq 2\int_{\Re^d}\|W \nabla f\|^2\, d\mu,
\ee
where the function $W$, taking values in $\Re^{d\times d}$, has the
meaning of a \emph{weight}. Clearly, one can consider (\ref{LSI}) as a
particular case of (\ref{WLSI}) with constant weight~$W$ equal to $\sqrt
{c}$ multiplied by the identity matrix. The problem of giving explicit
conditions on $\mu$ that ensure the log-Sobolev inequality or its
modifications is intensively studied in the literature, in particular,
because of numerous connections between these inequalities with measure
concentration, semigroup properties, and so on (see, e.g., \cite
{Ledoux}). Motivated by this general problem, in this paper, we propose
an approach that is based mainly on martingale methods and provides
explicit bounds for the entropy with the right-hand side given in a
certain integral form.

Our approach is motivated by the well-known fact that, on a path space
of a Brownian motion, the log-Sobolev inequality possesses a simple
proof based on fine martingale properties of the space (cf.\ \cite
{Cap_Hsu_ledoux,Gong_Ma}). We observe that a part of this proof
is, to a high extent, insensitive w.r.t.\ the structure of the
probability space; we formulate a respective martingale bound for the
entropy in Section~\ref{s2}. To apply this general bound on a
probability space of the form $(\Re^d, \mu)$, one needs a proper
martingale structure therein. In Section \ref{s3}, we introduce such a
structure in terms of a \emph{trimming filtration}, defined in terms of
a set of \emph{trimmed regions} in $\Re^d$. This leads to an integral
bound for the entropy on $(\Re^d, \mu)$. In Section \ref{s4}, we show
the way how this bound can be used to obtain a weighted log-Sobolev
inequality; this is made in the one-dimensional case $d=1$, although we
expect that similar arguments should be effective for the
multidimensional case as well; this is a subject of our further research.

\subsection{A martingale bound for the entropy}\label{s2}

Let $(\varOmega, \Ff, \P)$ be a probability space with filtration $\FF=\{
\Ff_t, t\in[0,1]\}$, which is right-continuous and complete, that is,
every $\Ff_t$ contains all $\P$-null sets from~$\Ff$. Let $\{M_t,t\in
[0,1]\}$ be a nonnegative square-integrable martingale w.r.t. $\FF$ on
this space, with c\`adl\`ag trajectories. We will use the following
standard facts and notation (see \cite{Ell}).

The martingale $M$ has unique decomposition $M=M^c+M^d$, where $M^c$ is
a continuous martingale, and $M^d$ is a purely discontinuous martingale
(see \cite{Ell}, Definition~9.20). Denote by $\la M^c \ra$ the \emph
{quadratic variation} of $M^c$, by
\[
[M]_t=\bigl\langle M^{c}\bigr\rangle_{t} +
\sum_{s \leq t} (M_s-M_{s-})^{2}
\]
the \emph{optional quadratic variation} of $M$, and by $\la M\ra$ the
\emph{predictable quadratic variation} of $M$, that is, the projection
of $[M]$ on the set of $\FF$-predictable processes. Alternatively, $\la
M\ra$ is identified as the $\FF$-predictable process that appears in
the Doob--Meyer decomposition for $M^2$, that is, the $\FF$-predictable
nondecreasing process $A$ such that $A_0=0$ and $M^2-A$ is a martingale.

For a nonnegative r.v. $\xi$, define its entropy by $\Ent\xi=\E\xi\log
\xi-\E\xi\log(\E\xi)$ with the convention $0\log0=0$.

\begin{thm}\label{t1} Let the $\sigma$-algebra $\Ff_0$ be degenerate.
Then for any nonnegative square-integrable martingale $\{M_t,t\in
[0,1]\}$ with c\`adl\`ag trajectories,
\[
\Ent M_1\leq\E\int_0^1
{1\over M_{t-}}d \langle M\rangle_{t}.
\]
\end{thm}
\begin{proof} Consider first the case where
\be\label{ass}
c_1\leq M_t\leq c_2, \quad t\in[0,1],
\ee
with some positive constants $c_1, c_2$. Consider a smooth function
$\varPhi$, bounded with all its derivatives, such that
\[
\varPhi(x)=x\log x, \quad x\in[c_1, c_2].
\]
Then by the It\^o formula (see \cite{Ell}, Theorem 12.19),
\begin{align*}
\varPhi(M_1)-\varPhi(M_0)&=\int
_0^1\varPhi'(M_{t-})\, dM_t+
{1\over2}\int_0^1\varPhi
''(M_{t-})\, d \bigl\langle M^c
\bigr\rangle_{t}\\
&\quad+\sum_{0< t\leq1} \bigl[\varPhi(M_{t})-
\varPhi(M_{t-})-\varPhi '(M_{t-})
(M_t-M_{t-}) \bigr].
\end{align*}
Clearly,
\[
\E\int_0^1\varPhi'(M_{t-})\, dM_t=0.
\]
Because $\Ff_0$ is assumed to be degenerate, $M_0=\E[M_1|\Ff_0]=\E M_1$
a.s., and hence
\begin{align*}
\Ent M_1&=\E \bigl(\varPhi(M_1)-
\varPhi(M_0) \bigr)\\
&={1\over2}\E\int
_0^1\varPhi ''(M_{t-})\, d
\bigl\langle M^c\bigr\rangle_{t}\\
&\quad+\E\sum_{0< t\leq1} \bigl[\varPhi(M_{t})-
\varPhi(M_{t-})-\varPhi '(M_{t-})
(M_t-M_{t-}) \bigr].
\end{align*}
For $x\in[c_1, c_2]$, we have $\varPhi'(x)=1+\log x$ and $\varPhi''(x)=1/x$.
Observe that
for any $x,\delta$ such that $x,x+\delta\in[c_1, c_2]$,
\[
\ba \varPhi(x+\delta)-\varPhi(x)-\varPhi'(x) \delta&=(x+\delta)
\log(x+\delta)-x\log x-\delta(1+\log x)
\\
&=(x+\delta)\log \biggl(1+{\delta\over x} \biggr)-\delta\leq(x+\delta )
{\delta\over x}-\delta=\frac{\delta^{2}}{x}. \ea
\]
Then
\[
\ba \Ent M_1&\leq{1\over2}\E\int_0^1
{1\over M_{t-}}d \bigl\langle M^c\bigr\rangle _{t}+
\E\sum_{0< t\leq1}{(M_t-M_{t-})^2\over M_{t-}} \leq\E\int
_0^1{1\over M_{t-}}d [M]_t.
\ea
\]
Because the process $M_{t-},\, t\in[0,1]$, is $\FF$-predictable, we have
\[
\E\int_0^1{1\over M_{t-}}d
[M]_t=\E\int_0^1
{1\over M_{t-}}d \la M\ra_t,
\]
which completes the proof of the required bound under assumption (\ref{ass}).

The upper bound in this assumption can be removed using the following
standard localization procedure. For $N\geq1$, define
\[
\tau_N=\inf\bigl\{t\in[0,1]: M_t\geq N\bigr\}
\]
with the convention $\inf\varnothing=1$. Then, repeating the above
argument, we get
\[
\Ent M_{\tau_N}\leq\E\int_0^{\tau_N}
{1\over M_{t-}}d \la M\ra_t\leq\E \int_0^1
{1\over M_{t-}}d \la M\ra_t.
\]
We have $M_{\tau_N}\to M_1,N\to\infty$ a.s. On the other hand, $\E
M_{\tau_N}^2\leq\E M_{1}^2$, and
\[
x\log x=o\bigl(x^2\bigr),\quad x\to+\infty.
\]
Hence, the family $\{M_{\tau_N}\log M_{\tau_N}, N\geq1\}$ is uniformly
integrable, and
\[
\Ent M_{\tau_N}\to\Ent M_{1}, \quad N\to\infty.
\]
Passing to the limit as $N\to\infty$, we obtain the required statement
under the assumption $M_t\geq c_1>0$. Taking $M_t+(1/n)$ instead of
$M_t$ and then passing to the limit as $n\to\infty$, we complete the
proof of the theorem.
\end{proof}

We further give two examples where the shown martingale bound for the
entropy is applied. In these examples, it would be more convenient to
assume that $t$ varies in $[0,\infty)$ instead of $[0,1]$; a respective
version of Theorem \ref{t1} can be proved by literally the same argument.

\begin{ex}[Log-Sobolev inequality on a Brownian path space; \cite
{Cap_Hsu_ledoux,Gong_Ma}]\label{ex1} Let $B_t,\, t\geq0$, be a
Wiener process on $(\varOmega, \Ff, \P)$ such that $\Ff=\sigma(B)$. Let $\{
\Ff_t\}$ be the natural filtration for $B$. Then for every $\zeta\in
L_2(\varOmega, \P)$, the following martingale representation is available:
\be\label{clark}
\zeta=\E\zeta+\int_0^\infty\eta_s\, dB_s
\ee
with the It\^o integral of a (unique) square-integrable $\{\Ff_t\}
$-adapted process $\{\eta_t\}$ in the right-hand side (cf. \cite{Clark}).
Take $\xi\in L_4(\varOmega, \P)$ and put $\zeta=\xi^2$ and
\[
M_t=\E[\zeta|\Ff_t]=\E\zeta+\int_0^t
\eta_s\, dB_s,\quad t\geq0.
\]
Then the calculation from the proof of Theorem \ref{t1} gives the bound
\[
\Ent\xi^2\leq{1\over2}\E\int_0^1
{1\over M_{t-}}d \bigl\langle M^c\bigr\rangle _{t}=
{1\over2}\E\int_0^1
{\eta^2_t\over M_t} dt={1\over2}\E\int_0^1
{\eta^2_t\over\E[\xi^2|\Ff_t]} dt.
\]
Note the extra term $1/2$, which appears because the martingale $M$ is
continuous.\vadjust{\eject}

Next, recall the Ocone representation \cite{Ocone_Clark} for the
process $\{\eta_t\}$, which is valid if $\zeta$ possesses the Malliavin
derivative $D\zeta=\{D_t\zeta, t\geq0\}$\emph{:}
\be\label{Ocone}
\eta_t=\E[D_t\zeta|\Ff_t], \quad t\geq0.
\ee
We omit the details concerning the Malliavin calculus, referring the
reader, if necessary, to \cite{nualart}. Because the Malliavin
derivative possesses the chain rule, we have
\[
\eta^2_t=4 \bigl(\E[\xi D_t\xi|
\Ff_t] \bigr)^2\leq4 \E\bigl[\xi^2\bigl|
\Ff_t\bigr] \E \bigl[(D_t\xi)^2\bigr|
\Ff_t\bigr],
\]
and consequently the following log-Sobolev-type inequality holds\emph{:}
\be\label{LSI_B}
\Ent\xi^2\leq2\E\int_0^1 \E\bigl[(D_t\xi)^2\big|\Ff_t\bigr]\, dt=2\E\|D\xi\|^2_H,
\ee
where $D\xi$ is considered as a random element in $H=L_2(0, \infty)$.
By a proper approximation procedure one can show that (\ref{LSI_B})
holds for every $\xi\in L_2(\varOmega, \P)$ that has a Malliavin
derivative $D\xi\in L_2(\varOmega, \P, H)$.
\end{ex}

The previous example is classic and well known. The next one apparently
is new, which is a bit surprising because the main ingredients therein
(the Malliavin calculus on the Poisson space and the respective
analogue of the Clark--Ocone representation (\ref{clark}), (\ref
{Ocone})) are well known (cf. \cite{Carlen_Pardoux,TsoiEl}).

\begin{ex}[Log-Sobolev inequality on the Poisson path space]\label{ex2}
Let $N_t$, $t\geq0$, be a Poisson process with intensity $\lambda$,
and $\Ff=\sigma(N)$. Denote by $\tau_k,\, k\geq1$, the moments of
consequent jumps of the process $N$, and by $\Ff_t=\sigma(N_s, s\leq
t)$, $t\geq0$, the natural filtration for $N$. For any variable of the form
\[
\xi=F(\tau_1, \dots, \tau_n)
\]
with some $n\geq1$ and some compactly supported $F\in C^1(\Re^n)$,
define the random element $D\xi$ in $H=L_2(0, \infty)$ by
\[
D\xi=-\sum_{k=1}^nF'_k(
\tau_1, \dots, \tau_n)1_{[0,\tau_k]}.
\]
Denote by the same symbol $D$ the closure of $D$, considered as an
unbounded operator $L_2(\varOmega, \P)\to L_2(\varOmega, \P, H)$. Then the
following analogue of the Clark--Ocone representation (\ref{clark}), (\ref{Ocone})
is available (\cite{TsoiEl}):
for every $\zeta$ that possesses the stochastic derivative $D\zeta$,
the following martingale representation holds\emph{:}
\[
\zeta=\E\zeta+{1\over\lambda}\int_0^\infty
\eta_{s}\, d\tilde N_s,
\]
where $\tilde N_t=N_t-\lambda t$ denotes the compensated Poisson
process corresponding to~$N$, and $\{\eta_t\}$ is the projection in
$L_2(\varOmega, \P, H)$ of $D\xi$ on the subspace generated by the $\{\Ff
_t\}$-predictable processes.

Proceeding in the same way as we did in the previous example, we obtain
the following
log-Sobolev-type inequality on the Poisson path space\emph{:}
\be\label{LSI_B}
\Ent\xi^2\leq{4\over\lambda^2}\E\|D\xi\|^2_H.
\ee
\end{ex}

\section{Trimmed regions on $\Re^d$ and associated integral bounds for
the entropy}\label{s3} Let $\mu$ be a probability measure on $\Re^d$
with Borel $\sigma$-algebra $\mathcal{B}(\Re^d)$. Our further aim is to
apply the general martingale bound from Theorem \ref{t1} in the
particular setting $
(\varOmega, \Ff, \P)=(\Re^d, \mathcal{B}(\Re^d), \mu)$. To this end, we
first \emph{construct} a filtration $\{\Ff_t, t\in[0,1]\}$.

In what follows, we denote $\Nf_\mu=\{A\in\Ff: \mu(A)=0\}$ (the class
of $\mu$-null Borel sets).

Fix the family $\{D_t, t\in[0,1]\}$ of closed subsets of $\Re^d$ such that:
\begin{itemize}
\item[(i)] $D_s\subset D_t,\, s\leq t$;
\item[(ii)] $D_0\in\Nf_\mu$, $\mu(D_t)<1$ for $t<1$, and $D_1=\Re^d$;
\item[(iii)] for every $t>0$,
\[
D_t\setminus \biggl(\bigcup_{s<t}D_s
\biggr)\in\Nf_\mu,
\]
and for every $t<1$,
\[
D_t =\bigcap_{s>t}D_s.
\]
\end{itemize}
We call the sets $D_t,\, t\in[0,1]$, \emph{trimmed regions}, following
the terminology used frequently in the multivariate analysis (cf. \cite
{KM_AnnSta}). Given the family $\{D_t\}$, we define the respective \emph
{trimmed filtration} $\{\Ff_t\}$ by the following convention. Denote
$Q_t=\Re^d\setminus D_t$. Then, by definition, a set $A\in\Ff$ belongs
to $\Ff_t$ if either $A\cap Q_t\in\Nf_\mu$ or $Q_t\setminus A\in\Nf_\mu$.

By the construction, $\FF=\{\Ff_t\}$ is complete. It is also clear
that, by property~(ii) of the family $\{D_t\}$, the $\sigma$-algebra
$\Ff_0$ is degenerate and, by property~(iii), the filtration $\FF$ is
continuous. Hence, we can apply Theorem \ref{t1}.

Fix a Borel-measurable function $g:\Re^d\to\Re^+$ that is
square-integrable w.r.t.~$\mu$. Consider it as a random variable on $
(\varOmega, \Ff, \P)=(\Re^d, \mathcal{B}(\Re^d), \mu)$ and define
\[
g_t=\E[g|\Ff_t], \quad t\in[0,1].
\]
Since the $\sigma$-algebra possesses an explicit description, we can
calculate every $g_t$ directly; namely, for $t>0$ and $\mu$-a.a.\ $x$,
we have
\be\label{g_t}
g_t(x)= \bigg\{
\begin{array}{@{}ll}
g(x), & x\in D_t, \\[2pt]
G_t, & x\in Q_t,
\end{array}
\ee
where we denote
\be\label{G_t}
G_t={1\over\mu(Q_t)}\int_{Q_t}g(y)\,\mu(dy).
\ee
Note that $\mu(Q_t)>0$ for $t<1$ and the function $G:[0,1)\to\Re^+$ is
continuous. In what follows, we consider the modification of the
process $\{g_t\}$ defined by (\ref{g_t}) for \emph{every} $x\in\Re^d$.
Its trajectories can be described as follows. Denote
\begin{align}\label{tau}
\tau(x)=\inf\{t:x\in D_t\};
\end{align}
then by property (iii) of the family $\{D_t\}$ we have $\tau(x)=\min\{
t:x\in D_t\}$, and by property (ii) we have $\tau(x)<1, x\in\Re^d, \tau
(x)=0\Leftrightarrow x\in D_0$. Then, for a fixed $x\in\Re^d$, we have
\[
g_t(x)=g(x)1_{t\geq\tau(x)}+G_t1_{t<\tau(x)}, \quad t
\in[0,1],
\]
which is a c\`adl\`ag function because $\{G_t\}$ is continuous on $[0,1)$.

\begin{thm}\label{t2} Let $g$ be a Borel-measurable function $g:\Re
^d\to\Re^+$, square-integrable w.r.t.\ $\mu$. Let $\{D_t\}$ be a
family of trimmed regions that satisfy \hbox{\emph{(i)--(iii)}}.

Then
\[
\Ent_\mu g\leq\int_{\Re^d}{(g(x)-G_{\tau(x)})^2\over G_{\tau(x)}}\mu(dx),
\]
where the functions $G$ and $\tau$ are defined by \eqref{G_t} and \eqref
{tau}, respectively.
\end{thm}

\begin{proof} We have already verified the assumptions of Theorem \ref
{t1}: the filtration $\{\Ff_t\}$ is complete and right continuous, and
the square-integrable martingale $\{g_t\}$ has c\`adl\`ag trajectories.
Because $g_1=g$ a.s. and $\Ff_0$ is degenerate, by Theorem \ref{t1} we have
the bound
\[
\Ent_\mu g\leq \E\int_0^1
{1\over g_{t-}}d \langle g\rangle_{t}.
\]
Hence, we only have to specify the integral in the right-hand side of
this bound. Namely, our aim is to prove that
\be\label{id}
\E\int_0^1{1\over g_{t-}}d \langle g\rangle_{t}=\int_{\Re
^d}{(g(x)-G_{\tau(x)})^2\over G_{\tau(x)}}\mu(dx).
\ee
First, we observe the following.

\begin{lem}\label{l1} Let $0<s<t< 1$, and let $\alpha$ be a bounded $\Ff
_{s}$-measurable random variable.
Then
\[
\E \bigl[\alpha \bigl(\langle g\rangle_{t}-\langle g
\rangle_{s} \bigr) \bigr]=\int_{D_{t}\setminus D_{s}}\alpha(x){
\bigl(g(x)-G_{\tau(x)}\bigr)^2}\,\mu(dx).
\]
\end{lem}
\begin{proof} By the definition of $\la g\ra$,
\[
\E \bigl[\alpha \bigl(\langle g\rangle_{t}-\langle g
\rangle_{s} \bigr) \bigr]=\E \bigl[\alpha \bigl( g_{t}^2-g_{s}^2
\bigr) \bigr]=\E \bigl[\alpha \bigl( \E \bigl(g_{t}^2\big|
\Ff_{s}\bigr)-g_{s}^2 \bigr) \bigr].
\]
We have
\[
g_{s}^2(x)= \left\{ %
\begin{array}{@{}ll}
g^2(x), & x\in D_{s}, \\
G_{s}^2, & x\in Q_{s},
\end{array}\right. %
 \quad
g_{t}^2(x)= \left\{ %
\begin{array}{@{}ll}
g^2(x), & x\in D_{t}, \\
G_{t}^2, & x\in Q_{t},
\end{array}\right. %
\]
and applying formula (\ref{g_t}) with $g=g^2_{t}$ and $t=s$, we get
\[
\E\bigl(g_{t}^2\big|\Ff_{s}\bigr)
(x)-g_{s}^2(x)= \left\{ %
\begin{array}{@{}ll}
0, & x\in D_{s}, \\
{H_{t, s}\over\mu(Q_s)}, & x\in Q_{s},
\end{array}\right. %
\]
\[
H_{t, s}= \biggl(\int_{D_{t}\setminus D_{s}}\bigl(g^2(x)-G_{s}^2
\bigr)\,\mu(dx)+\int_{Q_{t}}\bigl(G_{t}^2-G_{s}^2
\bigr)\,\mu(dx) \biggr).
\]
Because $\alpha$ is $\Ff_{s}$-measurable, it equals a constant on
$Q_{s}$ $\mu$-a.s. Denote this constant by $A$; then the previous
calculation gives
\[
\E \bigl[\alpha \bigl(\langle g\rangle_{t}-\langle g
\rangle_{s} \bigr) \bigr]=A H_{t,s}.
\]
Write $H_{t,s}$ in the form
\[
H_{t, s}=\int_{D_{t}\setminus D_{s}}g^2(x)\,\mu(dx)+
\mu(Q_t)G_t^2-\mu(Q_s)G_s^2.
\]
Denote
\[
\mu_t=\mu(Q_t), \qquad I_t=\int
_{Q_t}g\,d\mu;
\]
then
\[
\mu(Q_t)G_t^2=\mu_tG_t^2=
{I_t^2\over\mu_t}.
\]
Observe that the functions $\mu_t, t\in[0,1]$ and $I_t, t\in[0,1]$,
are continuous functions of a bounded variation and $\mu_t>0,\, t<1$. Then
\[
\ba \mu(Q_t)G_t^2-\mu(Q_s)G_s^2&=
\int_s^td \biggl({I_v^2\over\mu_v}
\biggr)=\int_s^t \biggl(-{I_v^2\over\mu_v^2}
 d\mu_v+2{I_v\over\mu_v} dI_v \biggr)
\\
&=\int_s^t \bigl(-G^2_v\,
 d\mu_v+2G_v\, dI_v \bigr). \ea
\]
It is easy to show that
\be\label{I_2}
-\int_s^t G^2_v\, d\mu_v=\int_{D_{t}\setminus D_{s}} G_{\tau(x)}^2\,\mu(dx).
\ee
Indeed, because $G$ is continuous on $[0, 1)$, the left-hand side
integral can be approximated by the integral sum
\[
\sum_{k=1}^m G^2_{v_k}(
\mu_{v_{k-1}}-\mu_{v_k}),
\]
where $s=v_0<\cdots<v_m=t$ is some partition of $[s,t]$. This sum equals
\[
\sum_{k=1}^m G^2_{v_k}
\mu(D_{v_k}\setminus D_{v_{k-1}}).
\]
For $x\in D_{v_k}\setminus D_{v_{k-1}}$, we have $\tau(x)\in[v_{k-1},
v_k]$. Hence, this sum equals
\[
\sum_{k=1}^m \int_{D_{v_k}\setminus D_{v_{k-1}}}G_{\tau(x)}^2
\,\mu (dx)=\int_{D_{t}\setminus D_{s}} G_{\tau(x)}^2\,\mu(dx)\vadjust{\eject}
\]
up to a residue term that is dominated by
\[
\sup_{u,v\in[s,t], |u-v|\leq\max_{k}(v_k-v_{k-1})}\big|G^2_u-G^2_v\big|
\]
and tends to zero as the size of the partition tends to zero. This
proves (\ref{I_2}). Similarly, we can show that
\[
\int_s^tG_v\, dI_v=-
\int_{D_{t}\setminus D_{s}} G_{\tau(x)} g(x)\,\mu(dx).
\]
We can summarize this calculation as follows:
\[
\E \bigl[\alpha \bigl(\langle g\rangle_{t}-\langle g
\rangle_{s} \bigr) \bigr]=A \int_{D_{t}\setminus D_{s}}{
\bigl(g(x)-G_{\tau(x)}\bigr)^2}\,\mu(dx).
\]
Because $\alpha(x)=A$ for $\mu$-a.a. $x\not\in D_s$, this completes
the proof.
\end{proof}

Let us continue with the proof of (\ref{id}). Assume first that $g\geq
c$ with some $c>0$. Then $g_t\geq c$, and consequently the process
$1/g_{t-}$ is left continuous and bounded. In addition, the function
$G_t=I_t/\mu_t$ is bounded on every segment $[0, T]\subset[0,1)$.

Fix $T<1$ and take a sequence $\{\lambda^n\}$ of \xch{dyadic}{diadic} partitions of $[0,T]$,
\[
\lambda^n=\bigl\{t^n_k, k=0, \dots,
2^n\bigr\},\quad t_k^n={Tk\over2^n},
\]
and define
\[
g^n_t=g_01_{t=0}+\sum
_{k=1}^{2^n}g_{t_{k-1}^n}1_{t\in(t_{k-1}^n, t_k^n]}.
\]
For every fixed $t>0$, the value $g^n_t$ equals the value of $g$ at
some (\xch{dyadic}{diadic}) point $t_n<t$, and $t_n\to t-$. Hence,
\[
{1\over g_t^n}\to{1\over g_{t-}}, \quad n\to\infty,
\]
pointwise. In addition, because of the additional assumption $g\geq c$,
this sequence is bounded by $1/c$. Hence, by the dominated convergence theorem,
\[
\E\int_0^T{1\over g_{t-}}d\la g
\ra_t=\lim_{n\to\infty}\E\sum_{k=1}^{2^n}
{1\over g_{t_{k-1}^n}} \bigl(\langle g\rangle_{t_k^n}-\langle g
\rangle_{t_{k-1}^n} \bigr);
\]
here we take into account that the point $t=0$ in the left-hand side
integral is negligible because $g_t\to\E g,\, t\to0+$, in $L_2$, and
consequently $\la g\ra_t\to0$, $t\to0+$, in $L_1$. By Lemma \ref{l1},
\[
\E\sum_{k=1}^{2^n}{1\over g_{t_{k-1}^n}}
\bigl(\langle g\rangle _{t_k^n}-\langle g\rangle_{t_{k-1}^n} \bigr)=\E
\sum_{k=1}^{2^n}\int_{D_{t_{k}^n}\setminus D_{t_{k-1}^n}}
{(g(x)-G_{\tau(x)})^2\over
G_{t_{k-1}^n}}\mu(dx);
\]
recall that $g_{t_{k-1}^n}(x)=G_{t_{k-1}^n}$ for $x\not\in
D_{t_{k-1}^n}$. Next, for $x\in D_{t_{k}^n}\setminus D_{t_{k-1}^n}$, we
have $|\tau(x)-t_{k-1}^n|\leq2^{-n}$. Because $G_t, t\in[0, T]$, is
uniformly continuous and separated\vadjust{\eject} from zero, and $G_{\tau(x)}, x\in
D_T$ is bounded, we obtain that
\[
\ba \E\int_0^T{1\over g_{t-}}d
\langle g\rangle_{t}&=\lim_{n\to\infty}\E \sum
_{k=1}^{2^n}\int_{D_{t_{k}^n}\setminus D_{t_{k-1}^n}}
{(g(x)-G_{\tau
(x)})^2\over G_{t_{k-1}^n}}\mu(dx)
\\
&= \int_{D_T}{(g(x)-G_{\tau(x)})^2\over G_\tau(x)}\mu(dx). \ea
\]
Taking $T\to1-$ and applying the monotone convergence theorem to both
sides of the above identity, we get (\ref{id}).

To remove the additional assumption $g\geq c$, consider the family
$g^n_t=g_t+1/n$. Then $\la g^n\ra= \la g \ra$, $g^n(x)-G^n_{\tau
(x)}=g(x)-G_{\tau(x)}$, $g_{t-}^n=g_{t-}+(1/n), G_{\tau(x)}^n=G_{\tau
(x)}+(1/n)$. Hence,
we can write (\ref{id}) for $g^n$, apply the monotone convergence
theorem to both sides of this identity, and get (\ref{id}) for $g$.
\end{proof}

\section{One corollary: a weighted log-Sobolev inequality on $\Re
$}\label{s4}

In this section, we show the way how the integral bound for the
entropy, established in Theorem \ref{t2}, can be used to obtain
weighted log-Sobolev inequalities. Consider a continuous probability
measure $\mu$ on $(\Re, \mathcal{B}(\Re))$ and denote by $p_\mu$ the
density of its absolutely continuous part. Fix a family of segments
$D_t=[a_t,b_t], t\in[0, 1)$, where $a_0=b_0$, the function $a_t$ is
continuous and decreasing to $-\infty$ as $t\to1-$, and the function
$b_\cdot$ is continuous and increasing to $+\infty$ as $t\to1-$. Then
the family
\[
D_t=[a_t,b_t],\quad t\in[0, 1), \qquad
D_1=\Re,
\]
satisfies the assumptions imposed before. Hence, Theorem \ref{t2} is applicable.

We call a function $f:\Re\to\Re$ \emph{symmetric} w.r.t. the family $\{
D_t\}$ if
\[
f(a_t)=f(b_t), \quad t\in[0,1).
\]
In the following proposition, we apply Theorem \ref{t2} to $g=f^2$,
where $f$ is smooth and symmetric.

\begin{prop}\label{prop1} Let $f:\Re\to\Re$ be a smooth function that
is symmetric w.r.t. the family $\{D_t\}$. Then
\[
\Ent_\mu f^2 \leq 4\int_{\Re} W(x)
\bigl(f'(x)\bigr)^2\, \mu(dx),
\]
where
\[
W(x)=V^2(x)\log \biggl({1\over\mu_{\tau(x)}} \biggr), \qquad V(x)=
\left\{ %
\begin{array}{@{}ll}
{\mu((-\infty,x))\over p_\mu(x)}, & x\leq a_0, \\
{\mu((x, \infty))\over p_\mu(x)}, & x>a_0.
\end{array}\right. %
\]
\end{prop}
\begin{proof}
Write
\[
g(x)-G_{\tau(x)}={1\over\mu_{\tau(x)}}\int_{Q_{\tau(x)}}
\bigl(g(x)-g(y)\bigr)\,\mu (dy)= {1\over\mu_{\tau(x)}}\int_{Q_{\tau(x)}}
\int_x^y g'(z)\, dz\,\mu(dy).
\]
Let us analyze the expression in the right-hand side. Observe that now
$Q_{\tau(x)}$ is the union of two intervals $(-\infty, a_{\tau(x)})$
and $(b_{\tau(x)}, +\infty)$. Denote
\[
Q_t^+=(b_t, \infty), \qquad Q_t^-=(-\infty,
a_t), \qquad\mu_t^{\pm}=\mu \bigl(Q_t^{\pm}
\bigr).
\]
The point $x$ equals either $a_{\tau(x)}$ or $b_{\tau(x)}$; hence,
because $g=f^2$ is symmetric,
\[
g(x)=g(a_{\tau(x)})=g(b_{\tau(x)}).
\]
Then we have
\[
\int_x^y g'(z)\, dz= \left\{
\begin{array}{@{}ll}
\int_{b_{\tau(x)}}^y g'(z)\, dz, & y\in Q_{\tau(x)}^+, \\
\int_{a_{\tau(x)}}^y g'(z)\, dz, & y\in Q_{\tau(x)}^-.
\end{array}\right. %
\]
Consequently,
\begin{align*}
\big|g(x)-G_{\tau(x)}\big|&\leq{1\over\mu_{\tau(x)}} \biggl[\int
_{Q_{\tau
(x)}^+}\int_{Q_{\tau(x)}^+, \tau(z)\leq\tau(y)} \big|g'(z)\big|\, dz
\,\mu (dy)\\
&\quad+ \int_{Q_{\tau(x)}^-}\int_{Q_{\tau(x)}^-, \tau(z)\leq\tau(y)}
\big|g'(z)\big|\, dz\,\mu(dy) \biggr].
\end{align*}
Using Fubini's theorem, we get
\[
\ba \big|g(x)-G_{\tau(x)}\big| &\leq{1\over\mu_{\tau(x)}} \biggl[\int
_{Q_{\tau
(x)}^+} \mu_{\tau(z)}^+ \big|g'(z)\big|\, dz+ \int
_{Q_{\tau(x)}^-} \mu_{\tau
(z)}^- \big|g'(z)\big|\, dz \biggr]
\\
&\leq{1\over\mu_{\tau(x)}} \int_{Q_{\tau(x)}}V(z)
\big|g'(z)\big|\,\mu(dz). \ea
\]
Because $g=f^2$ and hence $g'=2ff'$, by the Cauchy inequality we then have
\[
\ba &\bigl(g(x)-G_{\tau(x)}\bigr)^2\\
&\quad\leq4 \biggl({1\over\mu_{\tau(x)}}\int_{Q_{\tau
(x)}} \bigl( V(z)
f'(z) \bigr)^2\,\mu(dz) \biggr) \biggl(
{1\over\mu_{\tau
(x)}}\int_{Q_{\tau(x)}}\bigl(f(z)
\bigr)^2\,\mu(dz) \biggr)
\\
&\quad= 4 \biggl({1\over\mu_{\tau(x)}}\int_{Q_{\tau(x)}} \bigl( V(z)
f'(z) \bigr)^2\,\mu(dz) \biggr) G_{\tau(x)}. \ea
\]
Observe that
\[
z\in Q_{\tau(x)}\quad\Leftrightarrow\quad\tau(z)>\tau(x)\quad\Leftrightarrow \quad x\in
D_{\tau(z)}\setminus\{a_{\tau(z)}, b_{\tau(z)}\}.
\]
Hence, by Theorem \ref{t2} and Fubini's theorem we have
\[
\ba \Ent_{\mu}\bigl(f^2\bigr)&\leq4\int
_{\Re} \biggl({1\over\mu_{\tau(x)}}\int_{Q_{\tau(x)}}
\bigl( V(z) f'(z) \bigr)^2\mu(dz) \biggr)\,\mu(dx)
\\
&= 4\int_{\Re} \biggl(\int_{D_{\tau(z)}}
{\mu(dx) \over\mu_{\tau(x)}} \biggr) \bigl( V(z) f'(z)
\bigr)^2 \,\mu(dz). \ea
\]
Similarly to the proof of (\ref{I_2}), we can show that
\[
\int_{D_{t}}{\mu(dx) \over\mu_{\tau(x)}}=\log\mu_{s}
\big|_{s=0}^{s=t}=\log \biggl({1\over\mu_t} \biggr);
\]
the last identity holds because $\mu_{0}=1$. This completes the proof.
\end{proof}

Next, we develop a \emph{symmetrization} procedure in order to remove
the restriction for $f$ to be symmetric. For any $x\not=a_0$, one
border point of the segment $D_\tau(x)$ equals~$x$; let us denote
$s(x)$ the other border point. Denote also $s(a_0)=a_0$. Define the
$\sigma$-algebra $\hat\Ff$ of symmetric sets $A\in\Ff$, that is, such
that $x\in A\Leftrightarrow s(x)\in A$. For a function $f\in L_2(\Re,
\mu)$, consider its \emph{$L_2$-symmetrization}
\[
\hat f= \bigl(\E_\mu\bigl[f^2\big|\hat\Ff\bigr]
\bigr)^{1/2}.
\]
It can be seen easily that there exists a measurable function $p:\Re\to
[0,1]$ such that, for $\mu$-a.a. $x\in\Re$,
\[
(\hat f)^2(x)=p(x)f^2(x)+\bigl(1-p(x)
\bigr)f^2\bigl(s(x)\bigr)=\E_{\nu_x} f^2,
\]
where we denote
\[
\nu_x=p(x)\delta_x+\bigl(1-p(x)\bigr)
\delta_{s(x)}, \quad x\in\Re.
\]

We have
\[
\E_\mu f^2=\E_\mu(\hat f)^2
\]
and, consequently,
\begin{align*}
\Ent_\mu f^2-\Ent_\mu(\hat
f)^2&=\E_\mu f^2\log f^2-
\E_\mu(\hat f)^2\log(\hat f)^2\\
&=\E_\mu \bigl(\E_\mu\bigl[f^2\log
f^2- (\hat f)^2\log (\hat f)^2\big|\hat\Ff\bigr]
\bigr)\\
&=\int_{\Re} \bigl(\Ent_{\nu_x} f^2
\bigr) \,\mu (dx).\end{align*}

It is well known (cf. \cite{Ledoux}) that for a Bernoulli measure $\nu
=p\delta_1+q\delta_{-1}$ ($p+q=1$), the following discrete analogue of
the log-Sobolev inequality holds:
\[
\Ent_\nu f^2\leq C_p(Df)^2,
\qquad C_p= \left\{ %
\begin{array}{@{}ll}
pq{\log p-\log q\over p-q}, & p\not=q, \\
{1\over2}, & p=q,
\end{array}\right. %
\]
where we denote $Df=f(1)-f(-1)$. This yields the bound
\[
\ba \Ent_\mu f^2-\Ent_\mu(\hat
f)^2&\leq\int_{\Re} C_{p(x)} \bigl(f(x)-f
\bigl(s(x)\bigr) \bigr)^2\, \mu(dx)
\\
&=\int_{\Re} C_{p(x)} \biggl(\int_{D_{\tau(x)}}
f'(z)\, dz\biggr)\xch{}{)}^2\, \mu(dx). \ea
\]
By the Cauchy inequality,
\[
\biggl(\int_{D_{\tau(x)}} f'(z)\, dz
\biggr)^2\leq \biggl(\int_{D_{\tau
(x)}}
\bigl(f'(z)\bigr)^2{\mu_{\tau(z)}^{3/2}\over p^2_\mu(z)} \mu(dz)
\biggr) \biggl(\int_{D_{\tau(x)}} {\mu(dz)\over\mu_{\tau(z)}^{3/2}} \biggr),
\]
and, similarly to the proof of (\ref{I_2}), we can show that
\[
\int_{D_{\tau(x)}} {\mu(dz)\over\mu_{\tau(z)}^{3/2}}=2 \bigl(
\mu_{\tau
(x)}^{-1/2}-1 \bigr)<2\mu_{\tau(x)}^{-1/2}.
\]
This yields the following bound for the difference $\Ent_\mu f^2-\Ent
_\mu(\hat f)^2$, formulated in the terms of $f'$:
\[
\Ent_\mu f^2-\Ent_\mu(\hat f)^2
\leq2\int_{\Re} \bigl(f'(z)\bigr)^2
U(z)\,\mu(dz),
\]
\[
U(z)={\mu_{\tau(z)}^{3/2}\over p^2_\mu(z)}\int_{Q_{\tau(z)}} C_{p(x)}
{\mu(dx)\over\mu_{\tau(x)}^{1/2}}.
\]
Note that $C_p\leq1$ for any $p\in[0,1]$, and hence we have
\[
U(z)\leq2 \biggl({\mu_{\tau(z)}\over p_\mu(z)} \biggr)^2.
\]

Assuming that the bound from Proposition \ref{prop1} is applicable to
$\hat f$ (which is yet to be studied because $\hat f$ may fail to be
smooth), we obtain the following inequality, valid without the
assumption of symmetry of $f$:
\be\label{145}
\Ent_\mu f^2\leq\int_{\Re} \bigl(4W(x)\bigl((\hat f)'(x)\bigr)^2+2U(x)\bigl(f'(x)\bigr)^2
\bigr)\, \mu(dx).
\ee
The right-hand side of this inequality contains the derivative of $\hat
f$ and hence depends on the choice of the family of trimmed regions $\{
D_t\}$. We further give a~particular corollary, which appears when $\{
D_t\}$ is the set of \emph{quantile trimmed regions}. In what follows,
we assume $\mu$ to possess a positive distribution density~$p_\mu$ and
choose $\{D_t=[a_t, b_t]\}$ in the following way. Denote $q_v=F_\mu
^{-1}(v)$, that is, the quantile of $\mu$ of the level $v$, and put
\[
a_t=q_{1/2-t/2}, \quad b_t=q_{1/2+t/2}, \quad
t\in[0,1).
\]
In particular, $a_0=b_0=m$, the median of $\mu$. Denote also $\hat F_\mu
=\min(F_\mu, 1-F_\mu)$; observe that now we have
\[
\hat F_\mu(x)={1\over2}\mu_{\tau(x)}.
\]

\begin{thm}\label{t3} Let $\mu$ be a probability measure on $\Re$ with
positive distribution density $p_\mu$. Then, for any absolutely
continuous $f$, we have
\[
\Ent_\mu f^2\leq\int_{\Re}K(x)
\bigl(f'(x)\bigr)^2\, \mu(dx), \qquad K(x)= 8 \biggl(
{\hat F_\mu(x)\over p_\mu(x)} \biggr)^2 \biggl(\log{1\over2\hat F_\mu
(x)}+1
\biggr).
\]
\end{thm}
\begin{proof} First, observe that now the $L_2$-symmetrization of a
function $f$ has the form
\be\label{146}
\hat f(x)=\sqrt{{1\over2}\bigl(f^2(x)+f^2\bigl(s(x)\bigr)\bigr)}.
\ee
This identity is evident for functions $f$ of the form
$1_{(-\infty, F^{-1}(v))}$, $v\in(0, 1/2]$ and $1_{[F^{-1}(v), \infty
)}$, $v\in[1/2, 1)$, and then easily extends to general $f$.

Next, observe that
\be\label{147}
s(x)=F_\mu^{-1} \bigl(1-F_\mu(x) \bigr),
\ee
and because $F_\mu$ is absolutely continuous and strictly increasing,
$s(x)$ is absolutely continuous as well. Then $\hat f$ is absolutely
continuous with
\[
(\hat f)'(x)={f(x)f'(x)+f(s(x))f'(s(x))s'(x)\over\sqrt{2(f^2(x)+f^2(s(x)))}};
\]
here and below the derivatives are well defined for a.a.\ $x$. Using a
standard localization/approximation procedure, we can show that
Proposition \ref{prop1} is well applicable to any absolutely continuous
function. Hence, it is applicable to $\hat f$, and (\ref{145}) holds.

We have
\[
\ba \bigl((\hat f)'(x)\bigr)^2&\leq
{(f'(x))^2+(f(s(x))f'(s(x))s'(x))^2\over
f^2(x)+f^2(s(x))}
\\[3pt]
& \leq\bigl(f'(x)f(x)\bigr)^2+\bigl(f'
\bigl(s(x)\bigr)s'(x)\bigr)^2. \ea
\]
The function $W(x)$ in (\ref{145}) now can be rewritten as
\[
W(x)= \biggl({\hat F_\mu(x)\over p_\mu(x)} \biggr)^2\log
{1\over2\hat F_\mu(x)};
\]
hence,
\begin{align*} &\int_{\Re}W(x) \bigl(\bigl(\hat f'(x)\bigr)
\bigr)^2\, \mu(dx)\\[3pt]
&\quad\leq\int_{\Re}W(x)\bigl(f'(x)\bigr)^2\, \mu(dx)
+\int_{\Re}W(x) \bigl(f'\bigl(s(x)\bigr)
\bigr)^2\bigl(s'(x)\bigr)^2\, \mu(dx). \end{align*}
Let us analyze the second integral in the right-hand side. By (\ref{147}),
\be\label{148}
s'(x)=-{p_\mu(x)\over p_\mu(s(x))};
\ee
hence,
\[
\ba &\int_{\Re}W(x)\bigl(f'\bigl(s(x)\bigr)
\bigr)^2\bigl(s'(x)\bigr)^2\, \mu(dx)\\[3pt]
&\quad=\int_{\Re
}\bigl(f'\bigl(s(x)\bigr)
\bigr)^2 \biggl({\hat F_\mu(x)\over p_\mu(s(x))} \biggr)^2\log
{1\over
2\hat F_\mu(x)}p_\mu(x)\, dx. \ea
\]
Change the variables $y=s(x)$; observe that we have $x=s(y)$ and $\hat
F_\mu(x)=\hat F_\mu(y)$. Then we finally get\vadjust{\eject}
\begin{align*}
& \int_{\Re}W(x)\bigl(f'(y)
\bigr)^2\bigl(s'(x)\bigr)^2\, \mu(dx)\\
&\quad= \int_{\Re
}\bigl(f'\bigl(s(x)\bigr)
\bigr)^2 \biggl({\hat F_\mu(y)\over p_\mu(y)} \biggr)^2\log
{1\over
2\hat F_\mu(y)}p_\mu\bigl(s(y)\bigr)\, {p_\mu(y)\over p_\mu(s(y))}
dy
\\
&\quad=\int_{\Re}W(y) \bigl(f'(y)\bigr)^2
\, \mu(dy),
\end{align*}
and therefore
\[
\int_{\Re}W(x) \bigl(\bigl(\hat f'\bigr) (x)
\bigr)^2\, \mu(dx)\leq2\int_{\Re}W(x)
\bigl(f'(x)\bigr)^2\, \mu(dx).
\]

On the other hand, by identity (\ref{146}) we have now $C_{p(x)}=1/2$,
and the function $U(x)$ in (\ref{145}) can be rewritten as
\[
U(x)= \biggl({\mu_{\tau(x)}\over p_\mu(x)} \biggr)^2=4 \biggl(
{\hat F_\mu
(x)\over p_\mu(x)} \biggr)^2,
\]
which completes the proof of the statement.
\end{proof}

\end{document}